\documentclass[12pt,twoside,final,amsfonts]{amsart}
\usepackage[draft=false]{hyperref}
\usepackage{amssymb}
\usepackage[dvips,final]{graphicx,psfrag}
\usepackage{times,a4wide}

\usepackage[alphabetic,backrefs]{amsrefs}

\theoremstyle{plain}
\newtheorem*{theorem*}{Theorem}

\newtheorem{theorem}{Theorem}

\newtheorem*{proposition*}{Proposition}
\newtheorem{corollary}[theorem]{Corollary}
\newtheorem*{corollary*}{Corollary}
\newtheorem{lemma}[theorem]{Lemma}
\newtheorem*{lemma*}{Lemma}

\theoremstyle{definition}

\newtheorem*{remark*}{Remark}
\newtheorem*{remarks*}{Remarks}
\newtheorem*{conjecture*}{Conjecture}
\theoremstyle{definition}

\newcommand{\D}{\mathbb{D}}
\newcommand{\C}{\mathbb{C}}

\newcommand{\E}{\mathbb{E}}

\newcommand{\eps}{\epsilon}
\newcommand{\spj}{(\sum_{j=1}^n L_j)(\prod_{j=1}^n L_j)}
\newcommand{\sumj}{\sum_{j=1}^n L_j}
\newcommand{\prodj}{\prod_{j=1}^n L_j}

\newcommand{\Aut}{\operatorname{Aut}}

\newcommand{\Cov}{\operatorname{Cov}}

\newcommand{\Var}{\operatorname{Var}}

\title[Gaussian Analytic functions in the polydisk]
{Gaussian Analytic functions in the polydisk}

\author[Xavier Massaneda] {Xavier Massaneda}

\address{Departament de Matem\`atica Aplicada i An\`alisi,
Universitat  de Bar\-ce\-lo\-na, Gran Via 585, 08071-Bar\-ce\-lo\-na, Spain}
\email{xavier.massaneda@ub.edu}

\author[Bharti Pridhnani] {Bharti Pridhnani}

\address{Departament de Matem\`atica Aplicada i An\`alisi,
Universitat  de Bar\-ce\-lo\-na, Gran Via 585, 08071-Bar\-ce\-lo\-na, Spain}
\email{bharti.pridhnani@ub.edu}

\date{\today}

\keywords{}

\subjclass{}

\begin{document}

\thanks{Authors supported by the Generalitat de Catalunya (grant 2014 SGR 289)  and the Spanish Ministerio de Econom\'ia y Competividad (project MTM2011-27932-C02-01)}.

\begin{abstract}
We study hyperbolic Gaussian analytic functions in the unit polydisk of $\C^n$. Following the scheme previously used in the unit ball we first study the asymptotics of fluctuations of linear statistics as the directional intensities $L_j$, $j=1,\dots,n$ tend to $\infty$. Then we estimate the probability of large deviations of such linear statistics and use the estimate to prove a hole theorem. Our proofs are inspired by the methods of M. Sodin and B. Tsirelson for the one-dimensional case, and B. Shiffman and S. Zelditch for the study of the analogous problem for compact K\"ahler manifolds. 
\end{abstract}

\maketitle

\section*{Introduction}

This paper studies some properties of the zero sets of Gaussian analytic functions in the polydisk. The plan of the paper and the techniques of the proofs are the same as in \cite{BMP}, where the analogous problems in the unit ball were dealt with, so we will often just outline the proofs and refer to \cite{BMP} for the details.

Consider the unit polydisk in $\mathbb C^n$
\[
 \mathbb D^n=\bigl\{z\in\C^n : |z_j|<1, \ j=1,\dots, n\bigr\}
\]
and the normalised invariant measure 
\[
 d\nu_n(z)=\frac{dm(z)}{\pi^n [1-|z|^2]^2}\ ,
\]
where $dm$ stands for the Lebesgue measure. We simply write $d\nu$ when no confusion about the dimension can arise.

Here and throughout the paper we use the standard notations
\[
 [1-|z|^2]=\prod_{j=1}^n (1-|z_j|^2)
\]
and, for $L=(L_1,\dots,L_n)$, 
\[
 [1-|z|^2]^L=\prod_{j=1}^n (1-|z_j|^2)^{L_j}\ .
\]

Given a vector $L$ with $L_j>1$, $j=1,\dots, n$, consider the weighted Bergman space
\[
 B_L(\mathbb D^n)=\bigl\{f\in H(\D^n) : \|f\|_{n,L}^2:=c_{n,L} \int_{\D^n} |f(z)|^2 [1-|z|^2]^{L}  d\nu_n(z)< +\infty\bigr\} ,
\]
where $c_{n,L}=\prod_{j=1}^n(L_j-1)$ is chosen so that $\|1\|_{n,L}=1$.

Consider also the normalisation of the monomials $z^\alpha$ in the norm $\|\cdot\|_{n,L}$:
\[
 e_{\alpha}(z)=\prod_{j=1}^n \left(\frac{\Gamma(L_j+\alpha_j)}{\alpha_j!\Gamma(L_j)}\right)^{1/2} z^\alpha\ .
\]
As usual here we denote $z=(z_1,\dots,z_n)$ and use the multi-index notation $\alpha=(\alpha_1,\dots,\alpha_n)$, $\alpha !=\alpha_1!\cdots \alpha_n!$, $|\alpha|=|\alpha_1|+\cdots +|\alpha_n|$ and $z^\alpha=z_1^{\alpha_1}\cdots z_n^{\alpha_n}$.

The \emph{hyperbolic Gaussian analytic function} (GAF) of \emph{intensity} $L=(L_1,\dots, L_n)$ is defined as
\[
 f_L(z)=\sum_\alpha a_\alpha e_{\alpha}(z)=\sum_\alpha a_\alpha  \prod_{j=1}^n \left(\frac{\Gamma(L_j+\alpha_j)}{\alpha_j!\Gamma(L_j)}\right)^{1/2} z^\alpha \qquad z\in\D^n,
\]
where $a_\alpha $ are i.i.d. complex Gaussians of mean 0 and variance 1 (denoted $a_\alpha \sim N_{\C}(0,1)$).


The sum defining $f_L$ can be analytically continued to $L_j>0$, which we assume henceforth.

The probabilistic properties of the hyperbolic GAF are determined by its covariance kernel, which is given by (see \cite{ST1}*{Section 1}, \cite{Stoll}*{p.17-18}):
\begin{align*}
 K_L(z,w)&=\E[f_L(z)\overline{f_L(w)}]= \sum_\alpha  \prod_{j=1}^n \frac{\Gamma(L_j+\alpha_j)}{\alpha_j!\Gamma(L_j)} z^\alpha \bar w^\alpha=
 \prod_{j=1}^n \sum_{\alpha_j=0}^\infty   \frac{\Gamma(L_j+\alpha_j)}{\alpha_j! \Gamma(L_j)} (z_j\bar w_j)^{\alpha_j} \\
 &=
 \prod_{j=1}^n \frac 1{(1-z_j \bar w_j)^{L_j}}=
 \frac 1{[1-z \bar w]^{L}}\ .
\end{align*}

In this paper we follow the scheme of \cite{BMP} and study some statistical properties of the zero variety 
\[
 Z_{f_L}=\left\{z\in \D^n;\ f_L(z)=0\right\}\ .
\]
A main feature of the hyperbolic GAF is that the distribution of $Z_{f_L}$ is invariant under a large subgroup of the holomorphic automorphisms group $\Aut(\D^n)$. 
Consider the group $\mathcal A$ consisting of the automorphisms of the form
\[
 \phi_w^\theta(z)=\left(e^{i\theta_1} \frac{z_1-w_1}{1-\bar w_1 z_1}, \dots, e^{i\theta_n} \frac{z_n-w_n}{1-\bar w_n z_n}\right),\qquad w\in\D^n; \ \theta_j\in [0,2\pi)\ .
\]
We use the notation $\phi_w(z)$ in case $\theta_j=0$, $j=0,\dots,n$.

Any automorphism in $\Aut(\D^n)$ is the composition of an element of $\mathcal A$ with a permutation of the coordinates (see for instance \cite{Sha}*{Theorem 2, pag. 48}).

The transformations
\[
 T_w(f)(z)=\frac{[1-|w|^2]^{L/2}}{[1-\bar w z]^L} f(\phi_w^\theta(z))
\]
are isometries of $B_L(\D^n)$, hence the random zero sets $Z_{f_L}$ and $Z_{f_L\circ\phi_w^\theta}$ have the same distribution. More specifically, the distribution of the (random) integration current
\[
 [Z_{f_L}]=\frac {i}{2\pi}\partial\bar\partial\log |f_L|^2\ ,
\]
is invariant under the subgroup $\mathcal A$. In case $L_1=\cdots=L_n$ then $[Z_{f_L}]$ is invariant under the whole group $\Aut(\D^n)$.

The typical distribution of $Z_{f_L}$ is given by the Edelman-Kostlan formula (see \cite{HKPV}*{Section 2.4} and \cite{Sod}*{Theorem 1}): the so-called \emph{first intensity} of the GAF is 
\[
\E[Z_{f_L}] =\frac{i}{2\pi}\partial \overline{\partial}\log K_L(z,z)= \omega_L (z)\ ,
\]
where $\omega_L$ is the form
\[
\omega_L(z)=\sum_{j=1}^n \frac{L_j}{(1-|z_j|^2)^2} \frac{i}{2\pi} dz_j\wedge d\bar{z_j}\ .
\]
When $L_j=1$, $j=1,\dots,n$, we simply denote $\omega$. 
Notice that
\[
 \omega^n(z)= n!\bigwedge_{j=1}^n \frac{1}{(1-|z_j|^2)^2} \frac{i}{2\pi} dz_j\wedge d\bar{z_j}=\frac{n!}{\pi^n}\frac{dm(z)}{[1-|z|^2]^2}=n!\;  d\nu_n(z)\ ,
\]
which is invariant by $\Aut(\D^n)$ \cite{Stoll}*{p.19}.

In Section 1 we study the fluctuations of linear statistics as the $L_j$ tend to $\infty$. Let $\mathcal D_{(n-1,n-1)}$ denote the space of real-valued, compactly supported, $\mathcal C^2$ forms of bidegree $(n-1,n-1)$. For  $\varphi\in\mathcal D_{(n-1,n-1)}$, consider the integral of $\varphi$ over $Z_{f_L}$
\[
I_L(\varphi)=\int_{Z_{f_L}}\varphi=\int_{\D^n}  \varphi  \wedge [Z_{f_L}]  
\]
and note that the Edelman-Kostlan formula yields
\begin{equation}\label{E}
\E[I_{L}(\varphi)]=\int_{\D^n}\varphi\wedge \omega_L\ .
\end{equation}

\begin{theorem}\label{LS}
 Let $\varphi\in \mathcal D_{(n-1,n-1)}$ and let $D\varphi$ be the function defined by $\frac i{2\pi}\partial\bar\partial\varphi=D\varphi d\nu$. Then
 \[
  \Var [I_L(\varphi)]= \frac{\zeta(n+2)}{\prodj}\left(\int_{\D^n} (D\varphi)^2 d\nu\right)\Bigl[1+\textrm{O}(\sum_{j=1}^n\frac{\log L_j}{L_j})\Bigr].
 \]
\end{theorem}
Since $|\varphi\wedge \omega_L|\leq C(\varphi) \sum_{j=0}^n L_j$ (see \eqref{phiomega})  this shows a strong self-averaging of the volume $I_L(\varphi)$ (which increases with the dimension), in the sense that the variance is much smaller than the square of the typical values.

The same computations involved in the proof of this theorem show the asymptotic normality of $I_L(\varphi)$, i.e., that the distribution of
\[
\frac{I_L(\varphi)-\E[I_L(\varphi)]}{\sqrt{\Var [I_L(\varphi)]}}
\]
converges weakly to the (real) standard gaussian (Corollary~\ref{normality}), for each $\varphi$.


In Section 2, we estimate the probability of large deviations for $I_L(\varphi)$.

\begin{theorem}\label{smoothlargedeviations}
For all $\varphi\in \mathcal D_{(n-1,n-1)}$ and $\delta>0$, there exist $c>0$ and $L_j^0(\varphi,\delta,n)$, $j=1,\dots,n$, such that for all $L_j\geq L_j^0$,
\[
\mathbb P\bigl[|I_L(\varphi)-\E(I_L(\varphi))|>\delta \E(I_L(\varphi))\bigr]\leq e^{-c\spj}.
\]
\end{theorem}

\begin{remarks*}
 1. In case $n=1$ the result coincides with \cite{Bu-T13}*{Theorem 5.7} (see also \cite{BMP}*{Theorem 2}). Also, fixing $L_j$, $j\neq i$, and letting $L_i\to\infty$ we see that the exponent is of order $L_i^2$, which corresponds again to the one-dimensional case (for the coordinate $z_i$).

 2. When $L_j=L$ for all $j$, the exponent $(\sum_{j=0}^n L_j)(\prod_{j=0}^n L_j)$ is of order $L^{n+1}$, as  in the ball (see \cite{BMP}*{Theorem 2}) .
  
\end{remarks*}

Following the scheme of \cite{SZZ}*{pag.1994} we deduce a corollary that implies the upper bound in the hole theorem (Theorem~\ref{holethm} below). 

For a smooth compactly supported function $\psi$ in $\D^n$ consider the $(n-1,n-1)$-form
\[
 \varphi=\psi \frac{\omega^{n-1}}{(n-1)!}\ .
\]
In this case
\[
 \varphi (z) \wedge \omega_L (z) =\bigl(\sumj\bigr) \psi(z) d\nu(z)\ .
\]

Define
\[
 I_L(\psi)=\int_{Z_{f_L}}\psi\frac{\omega^{n-1}}{(n-1)!} =\int_{\mathbb D^n} \psi \wedge \frac{\omega^{n-1}}{(n-1)!} \wedge [Z_{f_L}]
\]
and note that \eqref{E} gives here
\[
 \E[I_L(\psi)]=\bigl(\sumj\bigr) \int_{\D^n} \psi\, d\nu\ .
\]

In particular, and for an open set $U\subset\D^n$ let $\chi_U$ denote its characteristic function and let  $I_L(U)=I_L(\chi_U)$. Then $\E[I_L(U)]=\bigl(\sum_{j=0}^n L_j\bigr) \nu(U)$. 

\begin{corollary}\label{conseqSLD}
Suppose that $U$ is an open set contained in a compact subset of $\D^n$. For all $\delta>0$ there exist $c>0$ and $L_j^0$ such that for all $L_j\geq L_j^0$,
\[
\mathbb P\Bigl[\Bigl|\frac{1}{\sum_{j=0}^n L_j}I_L(U)-\nu(U)\Bigr|>\delta\Bigr]\leq e^{-c\spj}.
\]
\end{corollary}

The proof of this is as in the ball (see \cite{BMP}*{Corollary 5}), so we skip it.

In the last Section we study the probability that $Z_{f_L}$ has a pseudo-hyperbolic hole of polyradius $r$. Given $w\in \D^n$ and $r=(r_1,\dots,r_n)$, $r_j \in (0,1)$ consider the pseudo-hyperbolic polydisk
\[
 E(w,r)=\Bigl\{z\in\D^n : \bigl|\frac{z_j-w_j}{1-z_j\bar w_j}\bigr|<r_j, j=1,\dots,n\Bigr\}\ .
\]
By the invariance of the distribution of the zero variety under the automorphisms $\mathcal A$, the probability that $Z_{f_L}$ does not intersect $E(w,r)$ is the same as the probability that $Z_{f_L}\cap E(0,r)=\emptyset$.

\begin{theorem}\label{holethm}
Let $r=(r_1,\dots,r_n)$, $r_j\in (0,1)$, be fixed. There exist $C_1=C_1(n,r)>0$, $C_2=C_2(n,r)>0$ and $L_j^0$ such that for all
$L_j\geq L_j^0$,
\[
e^{-C_{1}\spj} \leq \mathbb P\bigl[Z_{f_L}\cap E(0,r)=\emptyset\bigr]\leq e^{-C_{2}\spj}.
\]
\end{theorem}

A final word about notation. By $A\lesssim B$ we mean that there exists $C>0$ independent of the relevant variables of $A$ and $B$ for which $A\leq CB$. Then $A\simeq B$ means that $A\lesssim B$ and $B\lesssim A$.

\section{Linear statistics. Proof of Theorem~\ref{LS}}

The proof is as in \cite{BMP}*{Section 1} so we keep it short. The starting point is the following bi-potential expression of the variance:
\begin{align}\label{bipotential}
 \Var[I_L(\varphi)]&=
 \int_{\D^n}\int_{\D^n}  \rho_L(z,w) \frac i{2\pi}\partial\bar\partial\varphi(z) \frac i{2\pi}\partial\bar\partial\varphi (w)\\
 &=\int_{\D^n}\int_{\D^n}  \rho_L(z,w) D\varphi(z) D\varphi(w) d\nu(z) d\nu(w) \ , \nonumber
\end{align}
where $\rho_L(z,w)=4\Cov (\log|\hat f(z)|,\log|\hat f(w)|)$. By \cite{HKPV}*{Lemma 3.5.2}
\begin{align*}
  \rho_L(z,w)=\sum_{m=1}^\infty\frac{|\theta_L(z,w)|^{2m}}{m^2}\ ,
 \end{align*}
 where
\begin{equation}\label{norm-kernel}
\theta_L(z,w):=\frac{K_L(z,w)}{\sqrt{K_L(z,z)} \sqrt{K_L(w,w)}}=\frac{[1-|z|^2]^{L/2} [1-|w|^2]^{L/2}}{[1-\bar z w]^L}
\end{equation}
is the normalised covariance kernel of $f_L$.

For $\zeta,\xi\in\D$ let
\[
 \rho(\zeta,\xi)=\Bigl|\frac{\zeta-\xi}{1-\bar \zeta\xi}\Bigr|\ .
\]
Notice  that
\[
 |\theta_L(z,w)|^2=\prod_{j=1}^n \bigl(1-\bigl|\frac{z_j-w_j}{1-z_j\bar w_j}\bigr|^2\bigr)^{L_j}
 =\prod_{j=1}^n \bigl(1-\rho^2(z_j,w_j)\bigr)^{L_j}\ ,
\]
hence
\[
 |\theta_L(z,w)|^2=|\theta_L(\phi_w(z), 0)|^2
\]

We see next that only the near diagonal part of the double integral \eqref{bipotential} is relevant.
Let $\varepsilon_{j}=1/L_j^{2}$ and define
\[
 R_L^\epsilon=\Bigl\{(z,w)\in \D^n\times \D^n : 1-\rho^2(z_j,w_j)\geq \epsilon_j^{1/L_j}, j=1,\dots, n\Bigr\}\ .
\]

Split the integral into three parts
\begin{align}
  \Var[I_L(\varphi)]&=\int\limits_{(\D^n\times \D^n)\setminus R_L^\epsilon} \rho_L(z,w) D\varphi(z) D\varphi(w) d\nu(z) d\nu(w) \tag{I1} \\
  &\quad + \int_{R_L^\epsilon} \rho_L(z,w) (D\varphi(z)-D\varphi(w)) D\varphi(w) d\nu(z) d\nu(w) \tag{I2}\\
  &\quad + \int_{R_L^\epsilon} \rho_L(z,w) ( D\varphi(w))^2 d\nu(z) d\nu(w)\tag{I3}\ .
\end{align}

The bound for the first integral is a consequence of the estimate
\begin{equation}\label{est-rho}
 |\theta_L(z,w)|^2 \leq \rho_L(z,w)\leq 2 |\theta_L(z,w)|^2\ ,
\end{equation}
which follows immediately from $x\leq\sum_{m=1}^\infty\ x^m/m^2 \leq 2x$, $ x\in[0,1]$.

Then, by the definition of $R_L^\epsilon$,
\[
 |\textrm{I1}|\leq 2 \bigl(\prod_{j=1}^n \varepsilon_{j}\bigr) \int_{(\D^n\times \D^n)\setminus R_L^\epsilon} |D\varphi(z) D\varphi(w)| d\nu(z) d\nu(w)
 \leq  \frac 2{\prod_{j=1}^n L_j^2} \left(\int_{\D^n} |D\varphi(z)|\, d\nu(z)\right)^2\ .
\]

By the uniform continuity of $i\partial\bar\partial\varphi$ there exists a regular function $\eta(x_1,\dots,x_n)$ with $ \eta(1,\dots,1)=0$ and such that for all $z,w\in\D^n$,
\[
 |D\varphi(z)-D\varphi(w)|\leq \eta\bigl(1-\rho^2(z_1, w_1),\dots, 1-\rho^2(z_n, w_n)\bigr)\ .
\]
Since $\eta(x)\lesssim \|1-x\|$ for $x=(x_1,\dots, x_n)$, $x_j$ near 1, we see that for $(z,w)\in R_L^\epsilon$
\begin{align*}
|D\varphi(z)-D\varphi(w)|&\lesssim \|(\rho^2(z_1, w_1),\dots, \rho^2(z_n, w_n))\|
 \lesssim 
 \max_j\bigl|1-\varepsilon_j^{1/L_j}\bigr|\simeq\max_j \frac{\log L_j}{L_j}\ ,
\end{align*}

By the invariance by automorphisms of the measure $d\nu$, we get
\begin{align*}
 |\textrm{I2}|&\lesssim \Bigl( \max_j \frac{\log L_j}{L_j}\Bigr) \int\limits_{R_L^\epsilon\cap(supp\; \varphi\times supp\; \varphi)} [1-|\phi_z(w)|^2]^L\,  d\nu(z) d\nu(w)\\
&\lesssim \; \Bigl( \max_j \frac{\log L_j}{L_j}\Bigr)
\int_{supp\;\varphi}\Bigl(\int_{\{z: 1-|z_j|^2 \geq\varepsilon_{j}^{1/L_j}\ \forall j\}} [1-|z|^2]^L d\nu(z)\Bigr) d\nu(w)\\
&\lesssim \; \Bigl( \max_j \frac{\log L_j}{L_j}\Bigr) \int_{\{z: 1-|z_j|^2 \geq\varepsilon_{j}^{1/L_j}\ \forall j\}} [1-|z|^2]^L d\nu(z)\ .
\end{align*}

On the other hand, using again the invariance of $d\nu$, we see that
\begin{align*}
 \textrm{I3}&=
 \left(\int_{\D^n} (D\varphi(w))^2 d\nu(w)\right) \int_{\{z: 1-|z_j|^2 \geq\varepsilon_{j}^{1/L_j}\ \forall j\}} \rho_L(z,0) d\nu(z)\ .
\end{align*}
By \eqref{est-rho} we have thus $\textrm{I2}=\textrm{o}(\textrm{I3})$ and therefore
\begin{equation}\label{ls}
 \Var[I_L(\varphi)]=\textrm{I3} \bigl(1+\textrm{O}(\max_j \frac{\log L_j}{L_j})\bigr) \ .
\end{equation}

It remains to compute the second factor in I3:
\begin{align*}
 J:&=\int_{\{z: 1-|z_j|^2 \geq\varepsilon_{j}^{1/L_j}\ \forall j\}} \rho_L(z,0) d\nu(z)=\sum_{m=1}^\infty\frac 1{m^2} \int_{\{z: 1-|z_j|^2 \geq\varepsilon_{j}^{1/L_j}\ \forall j\}} [1-|z|^2]^{mL} d\nu(z)\\
 &=\sum_{m=1}^\infty\frac 1{m^2} \prod_{j=1}^n \int_{\{z_j: 1-|z_j|^2 \geq\varepsilon_{j}^{1/L_j}\}} (1-|z_j|^2)^{mL_j-2}\ \frac{dm(z_j)}{\pi}\ .
\end{align*}

Using the computations in \cite{BMP}*{Section 1} for $n=1$ we see that
\begin{align*}
 \int_{\{z_j: (1-|z_j|^2)^{L_j} \geq\varepsilon_{j}\}} (1-|z_j|^2)^{mL_j-2}\ \frac{dm(z)}{\pi}&=\int_0^{(1-\epsilon_j^{1/L_j})^{1/2}} (1-r^2)^{mL_j-2} 2r\ dr\\
 &=\int_{\epsilon_j^{1/L_j}}^1 s^{mL_j-2} ds=\frac 1{mL_j}\bigl[1+\textrm O(\frac 1{mL_j})\bigr]
\end{align*}
Therefore
\[
 J=\sum_{m=1}^\infty\frac 1{m^2} \prod_{j=1}^n \frac{1}{m L_j}\bigl[1+\textrm O(\frac 1{mL_j})\bigr]
 =\frac 1{\prodj}\; \zeta(n+2)\; \bigl[1+\textrm O(\max_j \frac 1{L_j})\bigr]
\]

This and \eqref{ls} give the stated result.

As an immediate consequence of the results of M. Sodin and B. Tsirelson and the previous computations we obtain the asymptotic normality of $I_L(\varphi)$. The proof is as in \cite{BMP}*{Corollary 5}, so we skip it.

\begin{corollary}\label{normality} As $L\to\infty$ the distribution of the normalised random variable
 \[
\frac{I_L(\varphi)-\E[I_L(\varphi)]}{\sqrt{\Var(I_L(\varphi))}}
\]
tends weakly to the standard (real) gaussian, for each $\varphi$.
\end{corollary}

\section{Large deviations. Proof of Theorem \ref{smoothlargedeviations}}

Applying Stokes' theorem, we have
\begin{align*}
I_{L}(\varphi)-\E\left[I_{L}(\varphi)\right]&
 =\int_{\D^n}\varphi\wedge \frac{i}{2\pi}\partial\overline{\partial}\log\frac{|f_L|^2}{K_L(z,z)}
=\int_{\D^n}\log\frac{| f_L|^2}{K_L(z,z)}\frac{i}{2\pi}\partial\overline{\partial}\varphi.
\end{align*}
Thus,
\[
|I_{L}(\varphi)-\E[I_{L}(\varphi)]|\leq \|D\varphi\|_\infty \int_{\text{supp}\varphi}\left|\log |\hat f_{L}(z)|^2\right|d\nu(z).
\]
Writing the form as
\[
 \varphi=\bigl(\frac i{2\pi}\bigr)^{n-1} \sum_{j,k=1}^n\varphi_{jk}\Bigl(\frac{dz_1}{1-|z_1|^2}\wedge\stackrel{\stackrel{\hat \j}{\smile}}{\dots}\wedge\frac{dz_n}{1-|z_n|^2}\wedge \frac{d\bar z_1}{1-|z_1|^2}\wedge\stackrel{\stackrel{\hat k}{\smile}}{\dots}\wedge\frac{d\bar z_n}{1-|z_n|^2} \Bigr)
\]
we see that
\[
 \varphi\wedge\omega_L=\sum_{j=1}^n\varphi_{jj} L_j\bigwedge_{k=1}^n\frac i{2\pi}\frac{dz_k\wedge d\bar z_k}{(1-|z_k|^2)^2}=\bigl(\sum_{j=1}^n L_j\varphi_{jj}\bigr) d\nu(z)\ ,
\]
and therefore
\begin{equation}\label{phiomega}
 \bigl|\varphi\wedge\omega_L\bigr|\lesssim c(\varphi) \bigl(\sum_{j=1}^n L_j \bigr)\ .
\end{equation}

This shows that the proof of Theorem \ref{smoothlargedeviations} will be completed as soon as we prove the following Lemma.

\begin{lemma}\label{mainlemmaSLD}
For any $\varphi\in\mathcal D_{(n-1,n-1)}$ and any $\delta>0$ there exists $c>0$ such that
\[
\mathbb P\left[ \int_{supp\ \varphi}\left|\log |\hat f_L(z)|^2 \right|d\nu(z)>\delta \sumj \right]\leq e^{-c\spj}.
\]
\end{lemma}
The key ingredient in the proof of this lemma is given by the following control on the average of
$\bigl|\log |\hat f_L|^2\bigr|$ over pseudo-hyperbolic polydisks.

\begin{lemma}\label{controlmean}
There exists a constant $c>0$ such that for a pseudo-hyperbolic polydisk $E=E(z_0,s)$, $z_0\in\D^n$, $s\in (0,1)$,
\[
\mathbb P\left[\frac{1}{\nu(E)}\int_{E}\left|\log |\hat f_L(\xi)|^2 \right|d\nu(\xi)>5 n(\sumj)\nu(E)^{1/n} \right]\leq e^{-c\spj}.
\]
\end{lemma}

Let us see first how this allows to complete the proof of Lemma~\ref{mainlemmaSLD}, and therefore of Theorem~\ref{smoothlargedeviations}.

\begin{proof}[Proof of Lemma~\ref{mainlemmaSLD}] Cover $K:=supp\; \varphi$ with pseudo-hyperbolic polydisks $E_k=E(\lambda_k,\eps)$, $k=1,\dots,N$ of fixed invariant volume $\nu(E_k)=\eta$ (to be determined later on). A direct estimate shows that $N\simeq \nu(K)/\eta$.

By Lemma~\ref{controlmean}, outside an exceptional event of probability $Ne^{-c\spj}\leq e^{-c'\spj}$,
\begin{align*}
\int_{K}\left|\log|\hat f_L(\xi)|^2\right|d\nu(\xi)&\leq \sum^{N}_{k=1}\int_{E_k}\left|\log |\hat f_L(\xi)|^2 \right|d\nu(\xi) \lesssim 
(\sumj) N \eta^{1+1/n}\\
& \simeq (\sumj) \nu(K) \eta^{1/n}.
\end{align*}
Choosing $\eta$ such that $\nu(K)\eta^{1/n}=\delta$ we are done.
\end{proof}

Now we proceed to prove Lemma~\ref{controlmean}. A first step is the following lemma.

\begin{lemma}\label{controlmax}
Fix $r=(r_1,\dots,r_n)\in (0,1)^n$, $z_0\in\D^n$ and $\delta>0$. There exists $c>0$ and $L_j^0=L_j^0(r,\delta)$ such that for all $L_j\geq L_j^0$
\begin{itemize}
 \item[(a)] $P \bigl[\max\limits_{E(z_0,r)}\log|\hat f_L(z)|^2<-\delta \sumj \bigr]\leq e^{-c\spj }$,
 \item[(b)] $P \bigl[\max\limits_{E(z_0,r)}\log|\hat f_L(z)|^2>\delta \sumj \bigr]\leq e^{-c\spj}$.
\end{itemize}
Combining both estimates $\mathbb P\bigl[\max\limits_{E(z_0,r)}\left|\log |\hat f_L(z)|^2\right|>\delta \sumj \bigr]\leq e^{-c\spj}$.

\end{lemma}

\begin{proof} By the invariance of the distribution of $\hat f$ under $\mathcal A$, it is enough to consider the case $z_0=0$.

(a) Consider the event
 \[
  \mathcal E_1=\left\{ \max_{E(0,r)}\log|\hat f_L(z)|^2<-\delta \sumj \right\}\ .
 \]
Note that
\begin{align*}
\log |\hat f_L(z)|^2=\log\frac{|f_L(z)|^2}{K_L(z,z)}& =\log|f_L(z)|^2-\log \prod_{j=1}^n \frac{1}{(1-|z_j|^2)^{L_j}},
\end{align*}
hence, by plurisubharmonicity,
\begin{align*}
\mathcal E_1& \subset \left\{ \max_{E(0,r)} \log |f_L(z)|^2 \leq \sumj \bigl(\log\frac{1}{1-r_j^2}- \delta\bigr)  \right\}.
\end{align*}

Therefore, for a suitable $\tilde\delta=\tilde\delta(r)$,
\[
 \mathcal E_1 \subset \left\{ \max_{E(0,r)} \log |f_L(z)|^2 \leq (1-2\tilde\delta) \bigl(\sumj \log\frac{1}{1-r_j^2}\bigr)\right\}\ 
\]
and the estimate of $\mathbb P[\mathcal E_1]$ will be done as soon as we prove the following lemma.

\begin{lemma}\label{maxlogfL}
For $0<\delta<1/2$ and a polyradius $r=(r_1,\dots,r_n)$ there exist $c=c(\delta,r)$ and $L_j^0=L_j^0(\delta,r)$ such that for all $L_j\geq L_j^0$
\[
\mathbb P\left[\max_{E(0,r)} \log|f_L(z)| \leq \bigl(\frac{1}{2}-\delta\bigr)\sumj \log\frac{1}{1-r_j^2} \right]\leq e^{-c\spj}\ ·
\]
\end{lemma}

\begin{proof} [Proof of Lemma~\ref{maxlogfL}]
Under the event we want to estimate we have
\[
\max_{|z|=r}|f_L(z)|\leq [1-r^2]^{-L\left(\frac{1}{2}-\delta\right)}\ .
\]
We shall see that this implies that some coefficients of the series of $f_L$ are necessarily ``small", something that only happens with a probability less than $\exp(-c\spj)$. 

Writing $f_L$ in Taylor series we see that
\[
a_{\alpha}=\left(\prod_{j=1}^n \frac{\alpha_j!\Gamma(L_j)}{\Gamma(L_j+\alpha_j)}\right)^{1/2}\frac{\partial^{\alpha}f_L(0)}{\alpha!}.
\]
With this and Cauchy's estimates
\[
\left| \frac{\partial^{\alpha}f_L(0)}{\alpha!}\right|\leq \frac{\max\limits_{E(0,r)} |f_L|}{r^\alpha}
\]
we have
\begin{equation*}\label{coeff}
|a_{\alpha}|\leq \left(\prod_{j=1}^n \frac{\alpha_j!\Gamma(L_j)}{\Gamma(L_j+\alpha_j)}\right)^{1/2}
\frac 1{r^\alpha [1-r^2]^{(\frac 12-\delta)L}} \ .
\end{equation*}

Stirling's formula 
\begin{equation*}\label{stirling}
 \Gamma(z)=\sqrt{\frac{2\pi}z} \left(\frac ze\right)^z \bigl[1+\textrm O(\frac 1z)\bigr]
\end{equation*}
yields 
\begin{align*}
\frac{\alpha_j!\Gamma(L_j)}{\Gamma(L_j+\alpha_j)}& \simeq \sqrt{2\pi}\alpha_j \sqrt{\frac{L_j+\alpha_j}{L_j \alpha_j}}
 \left(\frac{\alpha_j}{L_j+\alpha_j}\right)^{\alpha_j} \left(\frac{L_j}{L_j+\alpha_j}\right)^{L_j}\\
 &\leq \alpha_j 
 \left(\frac{\alpha_j}{L_j+\alpha_j}\right)^{\alpha_j} \left(\frac{L_j}{L_j+\alpha_j}\right)^{L_j}
\end{align*}

\textit{Claim:} For $\epsilon_j$ small and for the indices in
\[
 \mathcal I=\Bigl\{\alpha : L_j\bigl(\frac 1{(1-r_j^2)^{1-2\delta}}-1\bigr)\leq \alpha_j\leq \frac {L_j}{\frac{(1+\epsilon_j)^2}{r_j^2}-1}, \ j=1,\dots, n\Bigr\}
\]
the following estimate holds
\[
 |a_{\alpha}|^2\leq \prod_{j=1}^n (1+\epsilon_j)^{-\alpha_j}\ .
\]

\textit{Proof:}
The lower bound on $\alpha_j$ implies that
\[
 \left(\frac{L_j}{L_j+\alpha_j}\right)^{L_j/2} (1-r_j^2)^{-L_j (\frac 12-\delta)}\leq 1\ ,
\]
and therefore
\[
  |a_{\alpha}| \leq \prod_{j=1}^n \alpha_j^{1/2} \left(\frac{\alpha_j}{L_j+\alpha_j}\right)^{\alpha_j/2} \frac 1{r_j^{\alpha_j}}\ .
\]
The upper bound in $\mathcal I$ yields
\[
 \alpha_j^{1/2} \left(\frac{\alpha_j}{L_j+\alpha_j}\right)^{\alpha_j/2} \frac 1{r_j^{\alpha_j}} \leq (1+\epsilon_j)^{-\alpha_j/2}\ ,
\]
which gives the claim. $\square$

\medskip

Since for $\alpha\in\mathcal I$ we have $cL_j\leq \alpha_j\leq CL_j$, we deduce that the number of indices in $\mathcal I$ is of order $\prodj$. Therefore, letting $\xi\sim N_{\C}(0,1)$, and using the Claim, we have
\begin{align*}
 \mathbb P[\mathcal E_1]&\leq \mathbb P\Bigl[|a_\alpha|^2\leq \prod_{j=1}^n (1+\epsilon_j)^{-\alpha_j},\ \alpha\in\mathcal I\Bigr]\leq
 \left( P\Bigl[|\xi|^2\leq \prod_{j=1}^n (1+\epsilon_j)^{-\alpha_j}\Bigr] \right)^{c\prodj}\ .
\end{align*}

We finish by noticing that for $\alpha\in\mathcal I$
\begin{align*}
 P\Bigl[|\xi|^2\leq \prod_{j=1}^n (1+\epsilon_j)^{-\alpha_j}\Bigr]&=1-\exp\bigl(-e^{-\sum\limits_{j=1}^n \alpha_j\log(1+\epsilon_j)}\bigr)
 \simeq e^{-\sum\limits_{j=1}^n \alpha_j\log(1+\epsilon_j)}
 \simeq e^{-c\sum\limits_{j=1}^n L_j }\ .
\end{align*}
This finishes the proof of (a) in Lemma~\ref{controlmax}.
\end{proof}

 (b) Let now
 \begin{align*}
  \mathcal E_2:&=\left\{\max_{E(0,r)}\log|\hat f_L(z)|^2>\delta \sumj\right\}\\
  &=\left\{\max_{E(0,r)}\left[\log |f_L(z)|-\sum_{j=1}^n\frac{L_j}{2}\log\bigl(\frac{1}{1-|z_j|^2}\bigr)\right]>\frac{\delta}2 \sumj\right\}.
 \end{align*}
We estimate the probability of this event by controlling the coefficients of the series of $f_L$.
Let $C_j>0$ be constants to be determined later on. Split the series defining $|f_L|$ into two families of indices:
\begin{align*}
 I_1&=\bigl\{\alpha : \alpha_j\leq C_j \delta L_j,\ j=1,\dots,n\bigr\} \\
 I_2&=\mathbb N^n\setminus I_1=\bigl\{\alpha : \exists j\in\{1,\dots,n\}\ : \alpha_j> C_j \delta L_j\bigr\}
\end{align*}

Then, using Cauchy-Schwarz for the indices $I_1$ we see that
\begin{align*}
|f_L(z)|&\leq  \left(\sum_{\alpha\in I_1}|a_{\alpha}|^2\right)^{1/2} 
\left(\sum_{\alpha\in I_1} \prod_{j=1}^n \frac{\Gamma(L_j+\alpha_j)}{\alpha_j!\Gamma(L_j)} |z_j|^{2\alpha_j} \right)^{1/2}+
\sum_{\alpha\in I_2}  |a_{\alpha}|  \prod_{j=1}^n \left(\frac{\Gamma(L_j+\alpha_j)}{\alpha_j!\Gamma(L_j)}\right)^{1/2} r^{\alpha}\\
&\leq \left(\sum_{\alpha\in I_1}|a_{\alpha}|^2\right)^{1/2} \left(\prod_{j=1}^n \sum_{\alpha_j=0}^\infty \frac{\Gamma(L_j+\alpha_j)}{\alpha_j!\Gamma(L_j)} |z_j|^{2\alpha_j}\right)^{1/2} +\sum_{\alpha\in I_2}  |a_{\alpha}|  \prod_{j=1}^n \left(\frac{\Gamma(L_j+\alpha_j)}{\alpha_j!\Gamma(L_j)}\right)^{1/2} r^{\alpha}\\
&= \left(\sum_{\alpha\in I_1}|a_{\alpha}|^2\right)^{1/2} \sqrt{K_L(z,z)}+ \sum_{\alpha\in I_2}  |a_{\alpha}|  \prod_{j=1}^n \left(\frac{\Gamma(L_j+\alpha_j)}{\alpha_j!\Gamma(L_j)}\right)^{1/2} r^{\alpha}\\
&=:(I)+(II).\nonumber
\end{align*}

We shall estimate each part separately. First we shall see that, except for an event of small probability, $(II)$ is bounded (if $C_j$ are chosen appropiately). 

Fix $\gamma$ (to be determined later) and consider the event
\[
 A=\bigl\{ |a_\alpha|\leq e^{\gamma |\alpha|_\infty},\ \forall\alpha\in I_2 \bigr\}.
\]
Here $|\alpha|_\infty=\max_j \alpha_j$. We also use the notation
\[
 C_*=\min_j C_j\qquad ,\qquad L_*=\min_j L_j\qquad ,\qquad  L^*=\max_j L_j\ .
\]
Notice that $L^*\simeq\sumj$ and that for $\alpha\in I_2$
\[
 |\alpha|_\infty\geq\delta\min_j C_j L_j\geq\delta L_*  C_*\ .
\]
We split the indices $I_2$ in level sets
\[
 I_2^m=\bigl\{\alpha\in I_2 : |\alpha|_\infty=m\bigr\}\ .
\]
Observe that
\begin{equation}\label{i2m}
 m^{n-1}\leq \# I_2^m\leq n m^{n-1}
\end{equation}
and that for $\alpha\in I_2^m$ 
\[
 \prod_{j=1}^n \frac{\Gamma(L_j+\alpha_j)}{\alpha_j! \Gamma(L_j)}\leq \left(\frac{\Gamma(L^*+m)}{m! \Gamma(L^*)}\right)^n,
\]
since $\frac{\Gamma(L+n)}{n! \Gamma(L)}$ is increasing both in $n$ and $L$.

Under the event $A$, and denoting $r_0=\max_j r_j$, we have for $z\in E(0,r)$,
\begin{align*}
(II)& \leq 
\sum_{\alpha\in I_2} |a_\alpha|\left(\frac{\Gamma(L^*+|\alpha|_\infty)}{|\alpha|_\infty!\Gamma(L^*) }\right)^{n/2} r_0^{|\alpha|_\infty}
\leq \sum_{m\geq C_*\delta L_*} n m^{n-1} \left(\frac{\Gamma(L^*+m)}{m! \Gamma(L^*) }\right)^{n/2} (r_0 e^\gamma)^m\ .
\end{align*}

The asymptotics of the $\Gamma$-function
\begin{equation}\label{asymptoticsGamma}
 \lim_{m\to\infty}\frac{\Gamma(m+n)}{\Gamma(m) m^n}=1
\end{equation}
yields
\[
 (II)\leq \sum_{m\geq C_*\delta L_*}\left[\frac{m^{\frac 12-\frac 1n+\frac{L^*}2}}{\sqrt{\Gamma(L^*)}}\right]^n (r_0 e^\gamma)^m\ .
\]
By \cite{BMP}*{Lemma 10}, given $\epsilon>0$ there exists $c>0$ such that for $m\geq C_*\delta L_*$
\[
 \frac{m^{\frac 12-\frac 1n+\frac{L^*}2}}{\sqrt{\Gamma(L^*)}}\leq e^{\frac{\epsilon}m n}\ 
\]
and therefore
\[
 (II)\lesssim \sum_{m\geq C_*\delta L_*} (e^{\epsilon+\gamma}r_0)^m= \sum_{m\geq C_*\delta L_*} e^{-[\log\frac1{r_0}-(\epsilon+\gamma)]}\ .
\]
Choose $\epsilon=\gamma=\frac 14 \log\frac1{r_0}$, so that
\[
 (II)\lesssim \sum_{m\geq C_*\delta L_*} (r_0)^{m/2}\leq\frac 1{1-r_0^{1/2}}=:C(r)\ .
\]
With this we obtain the estimate
\[
 |f_L(z)|\leq \left(\sum_{\alpha\in I_1}|a_{\alpha}|^2\right)^{1/2} \sqrt{K_L(z,z)}+ C(r)\ .
\]
Under the event $\mathcal E_2$  we have then
\[
 e^{\frac{\delta}2 \sumj}<\frac{|f_L(z)|}{\sqrt{K_L(z,z)}}\leq \left(\sum_{\alpha\in I_1}|a_{\alpha}|^2\right)^{1/2}+
 \frac{C(r)}{\sqrt{K_L(z,z)}},
\]
and therefore, for $L_j$ big enough
\begin{equation}\label{equ}
 \sum_{\alpha\in I_1}|a_{\alpha}|^2 > e^{\frac{\delta}2 \sumj}\ .
\end{equation}

It remains to estimate the probability of this estimate, and to show that the event $A$ has ``big'' probability. The variables $|a_\alpha|^2$ are independent exponentials, hence
\[
\mathbb P[A]=\prod_{\alpha\in I_2}1-\mathbb P[|a_{\alpha}|\geq e^{\gamma|\alpha|_\infty}]
=\prod_{\alpha\in I_2}\left[1-e^{-e^{2\gamma |\alpha|_\infty}}\right].
\]

Since $x=e^{-e^{2\gamma |\alpha|_\infty}}$ is close to 0, we can use the estimate $\log(1-x)\simeq -x$. Thus, 
\begin{align*}
\log \mathbb P[A]&=\sum_{\alpha\in I_2}\log\left[1-e^{-e^{2\gamma |\alpha|_\infty}}\right]\simeq -
\sum_{\alpha\in I_2} e^{-e^{2\gamma |\alpha|_\infty}}\gtrsim -\sum_{m\geq C_*\delta L_*}n m^{n-1} e^{-e^{2\gamma m}}\\
&- \sum_{m\geq C_*\delta L_*} e^{-e^{\gamma m}}\simeq - e^{-e^{\gamma C_*\delta L_*}} .
\end{align*}

Choosing $C_j$ big enough so that, in addition to the previous conditions, $\gamma C_*>\max_j \log\frac{1}{1-r_j^2}$  we have
\[
 -e^{\gamma C_*\delta L_*}< -\frac 1{(1-r_0^2)^{\delta L_*}}
\]
and therefore
\[
 \log \mathbb P[A]\gtrsim -e^{-[1-r_0^2]^{-\delta L_*}}\ .
\]
On the other hand
\[
 \mathbb P(A\cap\mathcal E_2)\leq\mathbb P\bigl[\sum_{\alpha\in I_1} |a_{\alpha}|^2>e^{\frac{\delta}2 \sumj}\bigr]\ .
\]
Since the number of indices in $I_1$ is at most $M_L:=\prod_{j=1}^n C_j\delta L_j\simeq \prodj$ we have
\begin{align*}
\mathbb P[A\cap \mathcal E_2]&\leq \mathbb P\left[ \sum_{\alpha\in I_1} |a_{\alpha}|^2 \geq \frac 1{M_L} e^{\frac{\delta}2 \sumj}  \right]
= M_L e^{-\frac 1{M_L} e^{\frac{\delta}2 \sumj}}=e^{\sum_{j=1}^n\log(C_j\delta L_j)-\frac 1{M_L} e^{\frac{\delta}2 \sumj}}\ .
\end{align*}
For $L_j$ big enough
\[
 \frac 1{M_L} e^{\frac{\delta}2 \sumj}\geq e^{\frac{\delta}4 \sumj}\ 
\]
and therefore
\[
 \mathbb P[A\cap\mathcal E_2]\leq e^{-e^{\frac{\delta}8 \sumj}}\ .
\]
Also
\[
 \mathbb P[A^c\cap\mathcal E_2]\leq \mathbb P(A^c)\leq 1-e^{-e^{-\min\limits_j (1-r_j^2)^{-\delta L_j}}}\simeq e^{-\min\limits_j (1-r_j^2)^{-\delta L_j}} \ .
\]
All combined 
\[
\mathbb P[\mathcal E_2]=\mathbb P[A\cap \mathcal E_2]+\mathbb P[A^c\cap \mathcal E_2]\leq e^{-e^{\frac{\delta}8 \sumj}}+ e^{-\min\limits_j (1-r_j^2)^{-\delta L_j}}\leq e^{-c\spj} .
\]
\end{proof}

It remains to prove Lemma \ref{controlmean}. Before we proceed we need the following mean-value estimate of $\log |\hat f_L(\lambda)|^2$, which is obtained as \cite{BMP}*{Lemma 11}.

\begin{lemma}\label{invlemma}
Let $\lambda\in\D^n$ and $s\in (0,1)^n$ a polyradius. Then
\[
\log |\hat f_L(\lambda)|^2 \leq \frac{1}{\nu (E(\lambda,s))}\int_{E(\lambda,s)}\log |\hat f_L(\xi)|^2 d\nu(\xi)+\sumj \eps(s_j),
\]
where
\[
\eps(t)=\frac{1-t^2}{t^2}\int^{\frac{t^2}{1-t^2}}_{0} \log(1+x)dx\leq \frac{t^2}{1-t^2}.
\]
\end{lemma}

\begin{proof}[Proof of Lemma \ref{controlmean}]
According to Lemma~\ref{controlmax}(a), except for an exceptional event of probability $e^{-c\spj}$, there is $\lambda\in E:=E(z_0,s)$ such that
\[
-\sumj \frac{s_j^2}{1-s_j^2}<\log  |\hat f_L(\lambda)|^2.
\]
Therefore, using Lemma \ref{invlemma},
\[
-\sumj \frac{s_j^2}{1-s_j^2}<\frac{1}{\nu(E)}\int_{E}\log  |\hat f_L(\xi)|^2 d\nu(\xi)+\sumj \frac{s_j^2}{1-s_j^2} .
\]
Hence
\[
0 < \frac{1}{\nu(E)}\int_{E}\log |\hat f_L(\xi)|^2 d\nu(\xi) +2 \sumj \frac{s_j^2}{1-s_j^2},
\]
and separating the positive and negative parts of the logarithm we obtain:
\[
\frac{1}{\nu(E)}\int_{E}\log^{-}  |\hat f_L(\xi)|^2 d\nu(\xi)
\leq \frac{1}{\nu(E)}\int_{E}\log^{+} |\hat f_L(\xi)|^2 d\nu(\xi)+2 \sumj \frac{s_j^2}{1-s_j^2} .
\]
Finally, again by Lemma \ref{controlmax}, outside another exceptional event of probability $e^{-c\spj}$,
\begin{align*}
\frac{1}{\nu(E)}\int_{E}\left|\log  |\hat f_L(\xi)|^2 \right|d\nu(\xi)&\leq \frac{2}{\nu(E)}\int_{E}\log^{+}  |\hat f_L(\xi)|^2 d\nu(\xi)+2 
\sumj \frac{s_j^2}{1-s_j^2}\\
& \leq 2\max_{E}\log^{+}  |\hat f_L|^2+2 \sumj \frac{s_j^2}{1-s_j^2}\leq 5 (\sumj) \max_j \frac{s_j^2}{1-s_j^2} \\
&\leq 5(\sumj) n\nu(E)^{1/n}.
\end{align*}
\end{proof}

\section{The hole theorem. Proof of Theorem~\ref{holethm}}

The upper bound is a direct consequence of the results in the previous section. Letting $U=E(0,r)$ and applying Corollary \ref{conseqSLD} with $\delta\mu(U)$ instead of $\delta$ we get
\[
 \mathbb P\left[Z_{f_L}\cap E(0,r)=\emptyset\right]\leq\mathbb P\left[|I_L(U)- \nu(U)\sumj|>\delta \nu(U)\sumj \right]\leq e^{-C_{2}\spj}.
\]

The method to prove the lower bound is standard (see \cite{HKPV}*{Theorem 7.2.3} and \cite{ST1}): we shall choose three events forcing $f_L$ to have a hole $E(0,r)$ and then we shall see that the probability of such events is at least $e^{-C_{1}\spj}$. 

Our starting point is the estimate
\[
 |f_L(z)|\geq |a_0|-\left|\sum_{\alpha\in J_1} a_\alpha\prod_{j=1}^n \left(\frac{\Gamma(L_j+\alpha_j)}{\alpha_j! \Gamma(L_j)}\right)^{1/2} z_j^{\alpha_j}\right|-
 \left|\sum_{\alpha\in J_2} a_\alpha\prod_{j=1}^n \left(\frac{\Gamma(L_j+\alpha_j)}{\alpha_j! \Gamma(L_j)}\right)^{1/2} z_j^{\alpha_j}\right|\ ,
\]
where, for constants $C_j$ to be chosen later, 
\begin{align*}
 J_1&=\bigl\{\alpha\neq 0 : \alpha_j\leq C_j L_j,\ j=1,\dots,n \bigr\}, \\
 J_2&=\bigl\{\alpha  : \exists j\in\{1,\dots,n\},\ \alpha_j> C_j L_j, \bigr\}\ .
\end{align*}

The first event is
\[
 E_1:=\left\{\ |a_0|\geq 1\right\}\ ,
\]
which has probability
\[
\mathbb P[E_1]=\mathbb P[|a_0|^2\geq 1]=e^{-1}.
\]

The second event corresponds to the tail of the power series of $f_L$. Here we use the notations of the previous sections. Let, as in the previous section, 
\[
 J_2^m=\{\alpha\in J_2 ; |\alpha|_\infty=m\} .
\]
Then, for $z\in E(0,r)$,
\begin{align*}
 S_3:&=\left|\sum_{\alpha\in J_2} a_\alpha \prod_{j=1}^n \left(\frac{\Gamma(L_j+\alpha_j)}{\alpha_j! \Gamma(L_j)}\right)^{1/2} z_j^{\alpha_j}\right| \leq \sum_{\alpha\in J_2} |a_\alpha| \prod_{j=1}^n \left(\frac{\Gamma(L_j+\alpha_j)}{\alpha_j! \Gamma(L_j)}\right)^{1/2} r_j^{\alpha_j}\\
 &\leq\sum_{m\geq C_* L_*} \left(\frac{\Gamma(L_*+m)}{m!\Gamma(L_*)}\right)^{n/2} r_0^m \bigl(\sum_{\alpha\in J_2^m} |a_\alpha|\bigr) \ .
\end{align*}
By the same arguments as in \cite{BMP}*{Section 3}, for $ C_j$ big enough and for $m\geq C_*L_*$
\[
 \frac{\Gamma(L_*+m)}{m! \Gamma(L_*)}\leq \left[\frac{m^{L_*/m}}{\Gamma(L)^{1/m}}\right]^m .
\]
By Stirling, for $m\geq C_*L_*$,
\[
 \frac{m^{L_*/m}}{\Gamma(L)^{1/m}}\leq (eC_*)^{1/C_*} K^{\frac 1{2C_*}},
\]
where $K=\max\limits_{x>0} x^{1/x}= e^{-1/e}$.

Let $h(C):= (eC)^{1/C} K^{\frac 1{2C}}$ and take $C_*$ big enough so that 
\[
 \bigl( h(C_*)\bigr)^{n/2} r_0\leq (1-\delta)^2\ .
\]
Then
\[
 S_3\leq \sum_{m\geq C_*L_*} \bigl(h(C_*)\bigr)^{\frac{mn}2} r_0^m \bigl(\sum_{\alpha\in J_2^m} |a_\alpha|\bigr) \leq \sum_{m\geq C_*L_*} (1-\delta)^m\bigl(\sum_{\alpha\in J_2^m} |a_\alpha|\bigr).
\]
Now impose the event
\[
 E_2=\bigl\{|a_\alpha|\leq\frac mn\ \forall \alpha\in J_2^m,\ \forall m\geq C_* L_*\bigr\}\ .
\]
Under the event $E_2$, essentially by \eqref{i2m},
\[
 S_3\leq \sum_{m\geq C_* L_*}  (1-\delta)^m m^n
\]
and there exists $C_*$ big enough so that 
\[
 S_3\leq \frac 14\ .
\]

We shall see next that $\mathbb P[E_2]$ is big. We have
\begin{align*}
 \log\mathbb P(E_2)&= \sum_{m\geq cL_*} \sum_{\alpha\in J_2^m} \log\bigl(1-e^{-\frac 2n m}\bigr)\simeq 
 - \sum_{m\geq C_*L_*} \sum_{\alpha\in J_2^m} e^{-\frac 2n m}\\
 &\gtrsim \sum_{m\geq C_*L_*}  n m^{n-1} e^{-\frac 2n m} ,
\end{align*}
thus for $C_*$ big enough $\mathbb P(E_3)\geq 1/2$.

The third event takes care of the middle terms in the power series of $f_L$. Let
\[
E_3:=\left\{\ |a_\alpha|^2<\frac{[1-r^2]^{L}}{4\prod_{j=1}^n C_j L_j}\quad \forall \alpha\in J_1\right\}.
\]
Using Cauchy-Schwarz's inequality we get, as in previous computations:
\begin{align*}
S_2&:= \left|\sum_{\alpha\in J_1} a_{\alpha} \prod_{j=1}^n \left(\frac{\Gamma(L_j+\alpha_j)}{\alpha_j! \Gamma(L_j)}\right)^{1/2} z_j^{\alpha_j}\right|\leq
\left(\sum_{\alpha\in J_1}|a_{\alpha}|^2\right)^{1/2}
\left(\sum_{\alpha\in J_1} \prod_{j=1}^n  \frac{\Gamma(L_j+\alpha_j)}{\alpha_j! \Gamma(L_j)}  r_j^{2\alpha_j} \right)^{1/2}\\
& \leq \left(\sum_{\alpha\in J_1} |a_{\alpha}|^2\right)^{1/2} \left(\prod_{j=1}^n\sum_{\alpha_j=0}^\infty \frac{\Gamma(L_j+\alpha_j)}{\alpha_j! \Gamma(L_j)}  r_j^{2\alpha_j}  \right)^{1/2}
\leq \left(\sum_{\alpha\in J_1}|a_{\alpha}|^2\right)^{1/2}[1-r^2]^{-L/2}.
\end{align*}
Under the event $E_3$,
\[
\sum_{\alpha\in J_1}|a_{\alpha}|^2\leq \sum_{\alpha\in J_1}\frac{1}{4} [1-r^2]^{L}
=\frac{1}{16}[1-r^2]^{L},
\]
and therefore
\[
S_3\leq \frac{1}{2}.
\]
Since $1-e^{-x}\geq x/2$ for $x\in(0,1/2)$, we get
\begin{align*}
\mathbb P[E_2]&=\prod_{\alpha\in J_1}\Bigl\{1-\exp\bigl(-\frac{[1-r^2]^{L}}{4\prod_{j=1}^n C_j L_j}\bigr)\Bigr\}\geq \prod_{\alpha\in J_1} \frac{[1-r^2]^{L}}{8\prod_{j=1}^n C_j L_j}=\left(\frac{[1-r^2]^{L}}{8\prod_{j=1}^n C_j L_j}\right)^{\prod_{j=1}^n C_j L_j}\\
&\geq \exp\Bigl[-c\prodj\bigl(\sumj\log(\frac 1{1-r_j^2})+\log 8+\sum_{j=1}^n\log (C_j L_j)\bigr)\Bigr]\geq e^{-c\spj}\ .
\end{align*}

Finally,
\[
\mathbb P[E_1\cap E_2\cap E_3]\geq \mathbb P[E_1] \mathbb P[E_2] \mathbb P[E_3] \geq e^{-c\spj},
\]
and under this event $|f_L(z)|\geq 1-1/2-1/4 >0$.

\end{document}